\documentclass[11pt]{amsart}
\usepackage{amsmath, amsthm,,amscd, amsfonts, amssymb, hyperref, mathrsfs, textcomp, epsfig, a4wide, graphicx, IEEEtrantools}

\usepackage{tikz}

\newtheorem{lemma}{Lemma}

\newtheorem{theorem}{Theorem}
\newtheorem{question}{Question}
\newtheorem{prop}{Proposition}

\theoremstyle{definition}
\newtheorem{defn}{Definition}
\newtheorem*{notation}{Notation}
\theoremstyle{remark}
\newtheorem{remark}{Remark}

\newcommand{\Rt}{\mathbb{R}^3}
\newcommand{\Rf}{\mathbb{R}^4}

\begin{document}

\title{Exact Lagrangian caps of Legendrian knots} 

\author{Francesco Lin}
\address{Department of Mathematics, Massachusetts Institute of Technology} 
\email{linf@math.mit.edu}

\begin{abstract} 
We prove that any Legendrian knot in $(S^3,\xi_{std})$ bounds an exact Lagrangian surface in $\mathbb{R}^4\setminus B^4$ after a sufficient number of stabilizations. In order to show this, we construct a family combinatorial moves on knot projections with some additional data that correspond to Lagrangian cobordisms between knots.
\end{abstract}
\maketitle

\section*{Introduction}
Given a Legendrian knot $K$ in the contact manifold $(S^3,\xi_{std})$, it is natural to ask whether or not it bounds a Lagrangian surface $L$ in its standard (convex) filling $(B^4, \omega_{std})$. In  \cite{Cha} it is shown that the existence of such a surface (which we call a \textit{Lagrangian filling} of $K$) implies the identities
\begin{IEEEeqnarray*}{c}
\mathrm{tb}(K)=-\chi(L)\\
 \mathrm{rot}(K)=0
\end{IEEEeqnarray*}
for the classical invariants of the Legendrian knot, and combining these with the slice Thurston-Bennequin inequality (\cite{R})
\begin{equation*}
\mathrm{tb}(K)+|\mathrm{rot}(K)|\leq 2g_4(K)-1,
\end{equation*}
where $g_4(K)$ is the $4$-ball genus of $K$, it follows that such a surface $L$ indeed realizes the $4$-ball genus of $K$ and $K$ has maximal Thurston-Bennequin invariant within its topological class. Hence there are very restrictive limitations on the existence of such surfaces, even if we do not impose any exactness condition.
\par
On the other hand, the same technique tells us much less about \textit{Lagrangian caps} of $K$, i.e. Lagrangian surfaces $L$ properly embedded in $\Rf\setminus B^4$ bounding the given knot. This is because the constraints on the classical invariants in this case are
\begin{IEEEeqnarray*}{c}
\mathrm{tb}(K)=\chi(L)\\
 \mathrm{rot}(K)=0,
\end{IEEEeqnarray*}
and by applying the slice Thurston-Bennequin inequality again one obtains the much less restrictive conditions
\begin{IEEEeqnarray*}{c}
g(L)\geq 1\\
\mathrm{tb}(K)\leq-1.
\end{IEEEeqnarray*}
Indeed, the aim of the present paper is to prove the following result.
\begin{theorem}Given any Legendrian knot $K$ in $(S^3,\xi_{std})$, one can stabilize it a sufficient number of times so that the final result admits an exact Lagrangian cap.
\end{theorem}

More precisely, the theorem tells us the following. We say that a Lagrangian $L$ with Legendrian boundary $K$ is \textit{exact} if the pull-back of the primitive of the symplectic form $\lambda_{std}$ on $L$ has a primitive vanishing along $K$. This is equivalent to say that $\lambda_{std}\lvert_L$ has integral zero along every closed curve on $L$ and every arc with boundary on $K$. Moreover, recall that a \textit{stabilization} of a Legendrian knot $K$ is a Legendrian knot obtained by performing one of the local moves of Figure $1$ in the front projection. The Legendrian isotopy type of the stabilized knot only depends on the effect of the stabilization on the rotation number (which is $\pm1$). Then the theorem says that for every Legendrian knot $K$ there is an $n(K)$ such that a knot obtained by stabilizing $K$ more that $n(K)$ times and such that $\mathrm{rot}(K)=0$ admits an exact Lagrangian cap.

\begin{figure}[here]
  \centering
\def\svgwidth{0.8\textwidth}
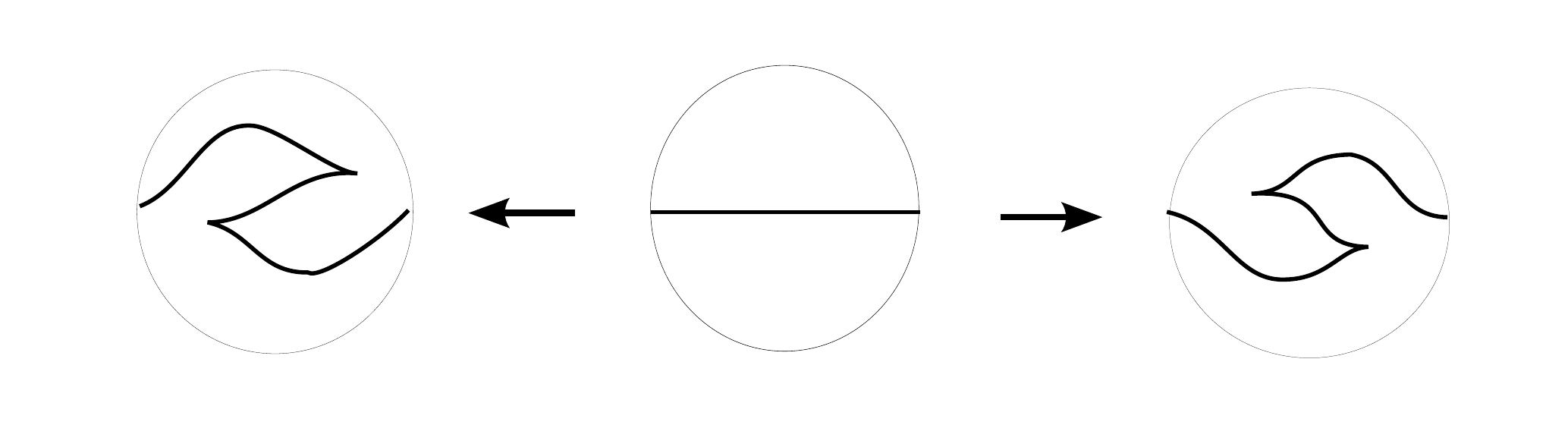
    \caption{Stabilization of a Legendrian knot.}
\end{figure} 
\begin{remark}The higher dimensional analogue of the problem is governed by an $h$-principle, see \cite{EM}. In \cite{Riz}, the existence of such caps is used to show the existence of (non orientable) closed embedded exact Lagrangians in $\mathbb{C}^2$ which are not uniruled and have infinite Gromov width.
\end{remark}

Here, and in the rest of the paper, we will work in the equivalent setting (for our problem) of $(\Rt, \xi_{std})$ where $\Rt$ has coordinates $(x,y,z)$, $\xi_{std}=\mathrm{ker}(\alpha_{std})$ with $\alpha_{std}=dz+xdy$ (notice that with this convention the front projection is the projection on the $yz$-plane while the Lagrangian projection is the projection to the $xy$-plane, rotated clockwise by $90$ degrees.). This means that we are interested in searching Lagrangian surfaces in the symplectization $(\mathbb{R}_t^{\geq0}\times\Rt, d(e^t\alpha_{std}))$ where the first factor is
\begin{equation*}\mathbb{R}_t^{\geq0}=\{t\in \mathbb{R}|t\geq0\}.
\end{equation*}
On the other hand, it will be more convenient for us to work with surfaces in
\begin{IEEEeqnarray*}{c}
\Rf_+=\{(x',y',t',z')| t'\geq0\}\subset (\Rf,\omega_{std})\\
\omega_{std}=dx'\wedge dy'+ dt'\wedge dz',
\end{IEEEeqnarray*}
with boundary condition on the hyperplane $\{t'=0\}$. In this case, our preferred primitive for the symplectic form will be the one form
\begin{equation*}
\lambda_{std}=x'dy'+(t'+1)dz'.
\end{equation*}
These models are indeed equivalent for our purposes via the symplectomorphism
\begin{IEEEeqnarray*}{c}
\Phi: \mathbb{R}_t^{\geq0}\times\Rt\rightarrow \Rf_+\\
(x,y,t,z)\mapsto (e^tx,y,e^t-1,z),
\end{IEEEeqnarray*}
which is a strict contactomorphism on the boundary and preserves the fixed primitives of the symplectic forms. We will always work with the latter, and drop the apices in the notation.
\\
\par
In order to prove Theorem $1$, we study the set of \textit{Lagrangian diagrams} of a surface $L\subset \Rf_+$. A Lagrangian diagram is simply the Lagrangian projection of a generic slice on which the $t$-coordinate is constant together with the assignment to each component of the complement of the knot projection of its volume with respect to the standard area form $dx\wedge dy$. We will introduce a set of combinatorial moves on such diagrams (up to some identifications) and prove that they can always be realized by Lagrangian cobordisms. These constructions are inspired by the techniques developed in \cite{Sau} to show the existence of immersed exact Lagrangians in $\mathbb{C}^2$ with few self-intersections.
\\
\par
With this in hand, the strategy to prove Theorem $1$ is the following. First, we prove that after a sufficient number of stabilizations each knot is exactly cobordant to the Legendrian unknot in Figure $2$, which we denote $U_0$. Then we will use the combinatorial moves to explicitly construct an exact Lagrangian cap of $U_0$. Also, we prove that if a knot $K$ admits a cap (hence $\mathrm{rot}(K)=0$) then also its double stabilization with rotation number zero admits a cap, justifying the ``sufficiently many stabilization" statement in Theorem $1$.
\begin{figure}[here]
\centering
\def\svgwidth{0.8\textwidth}
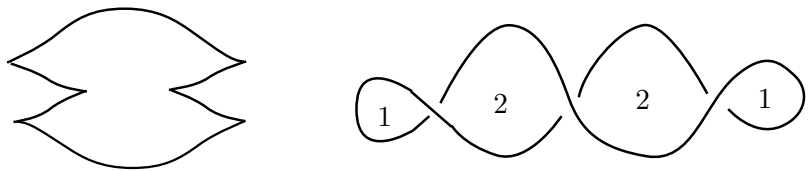
     \caption{A front projection and a Lagrangian diagram of the knot $U_0$.}
\end{figure}
\begin{figure}[here]
  \centering
\def\svgwidth{0.8\textwidth}
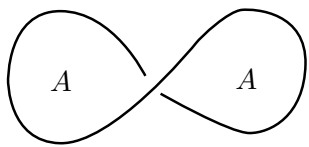
     \caption{A Lagrangian diagram of $U_1$.}
\end{figure}

It is interesting to point out that while in the first step we can restrict ourselves to Lagrangians such that the generic slice where the $t$ coordinate is constant is Legendrian, in order to construct a cap we necessarily have to rely on a wider class of Lagrangians. In fact the slices with $t$ constant just before the cap closes are not Legendrian knots, but have a Lagrangian diagram $U_1$ as in Figure $3$.
\\
\par
This is the plan for the rest of the paper. In Section $1$ we show how to reduce the problem of constructing a cap to the single knot $U_0$. In Section $2$ we define Lagrangian diagrams and construct the combinatorial moves on them corresponding to Lagrangian cobordisms. In Section $3$ we prove an exactness condition for the cobordism in terms of the Lagrangian diagrams from which it arises. Finally in Section $4$ we explicitly construct an exact Lagrangian cap for the knot $U_0$. In Section $5$ we discuss some interesting problems that arise.
\\
\par
\textit{Acknowledgments. }The author would like express his gratitude to Emmy Murphy for suggesting the problem and for all the very helpful and interesting discussions throughout all the time the present paper has been developed.

\vspace{1cm}
\section{Reduction to $U_0$}
In this section we prove the following result.
\begin{prop}
Given any Legendrian knot $K$, there is a Legendrian knot $K'$ obtained from $K$ via stabilizations such that there  exists an exact Lagrangian cobordism from $K'$ to the Legendrian unknot $U_0$ of Figure $2$.
\end{prop}
We will work only using front projections and rely on the moves depicted in Figure $4$, which already appear for example in \cite{EHK}.

\begin{figure}[here]
  \centering
\def\svgwidth{0.8\textwidth}
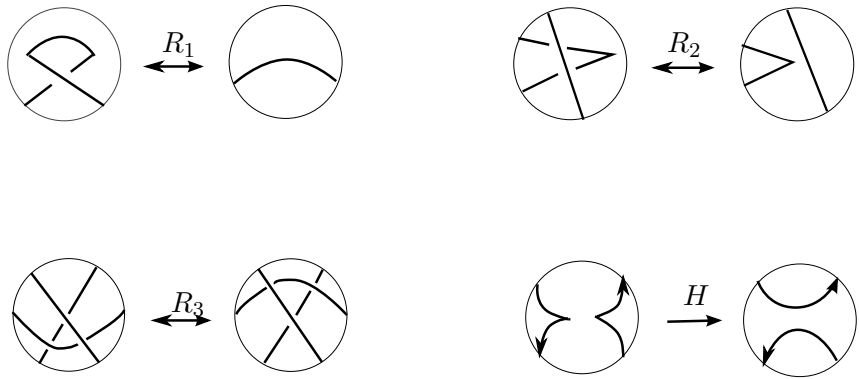
     \caption{Moves on front projections.}
\end{figure}

The moves $R_1,R_2$ and $R_3$ are the Legendrian Reidemester moves, and the fact that they can be realized by a Lagrangian cobordism follows from a result by Eliashberg and Gromov (\cite{EG}) stating that every Legendrian isotopy can be indeed realized in this way. On the other hand, the handle attachment move $H$ can be realized starting from Lagrangians in the first jet space $J^1\mathbb{R}^2$ by means of the following more general construction.
\begin{figure}[here]
  \centering
   \def\svgwidth{0.8\textwidth}
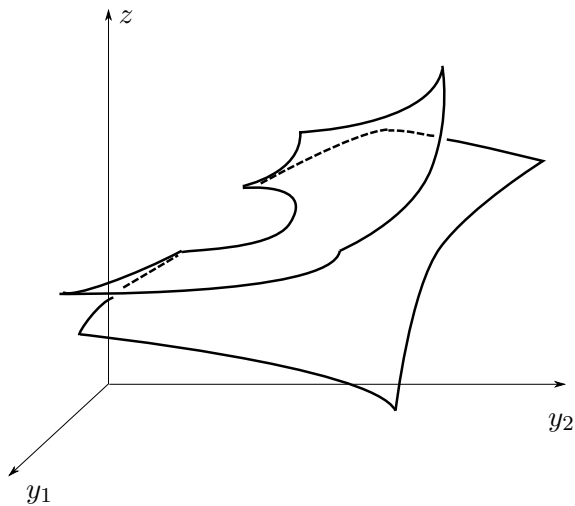     \caption{A front projection giving rise to the handle attachment move $H$.}
\end{figure}

Any front projection (with coordinates $(z,y_1,y_2))$ with boundary in $y_2=\{a,b\}$ such that:
\begin{itemize}
\item there are no points with the same $y_i$ coordinates and parallel tangent planes at those points;
\item the length of the Reeb chords in the slices with $y_2$ constant is increasing at the boundary,
\end{itemize}
corresponds to an exact embedded Lagrangian cobordism between the Legendrian knots arising as the slices $y_2=\{a,b\}$. In particular, the front projection in Figure $5$ corresponds to the move $H$.
\begin{remark}
Notice that in both constructions one does not have any control on the time interval in which the cobordism is defined. For example, in the jet spaces construction a non-trivial change of coordinates (depending on the front projection itself) is made in order to encode in a simple way the Legendrian boundary condition.
\end{remark}
We are now ready to describe how to construct a Lagrangian cobordism to $U_0$.
\begin{proof}[Proof of Proposition 1]
In our description we will add stabilizations along the way, but it will be clear from the construction that all of them can be performed at the beginning without affecting the final result. The proof is divided in various steps.
\\
\par
\textit{Step 1. }We show that up to stabilizations, every Legendrian knot is cobordant to one such that every component is an unknot and in the front projection each component is contained in a disk disjoint from the other ones. To do so, we perform a sequence of local moves near each crossing. There are essentially two cases up to orientation reversing (see Figure $6$).
\begin{figure}[here]
  \centering
\def\svgwidth{0.8\textwidth}
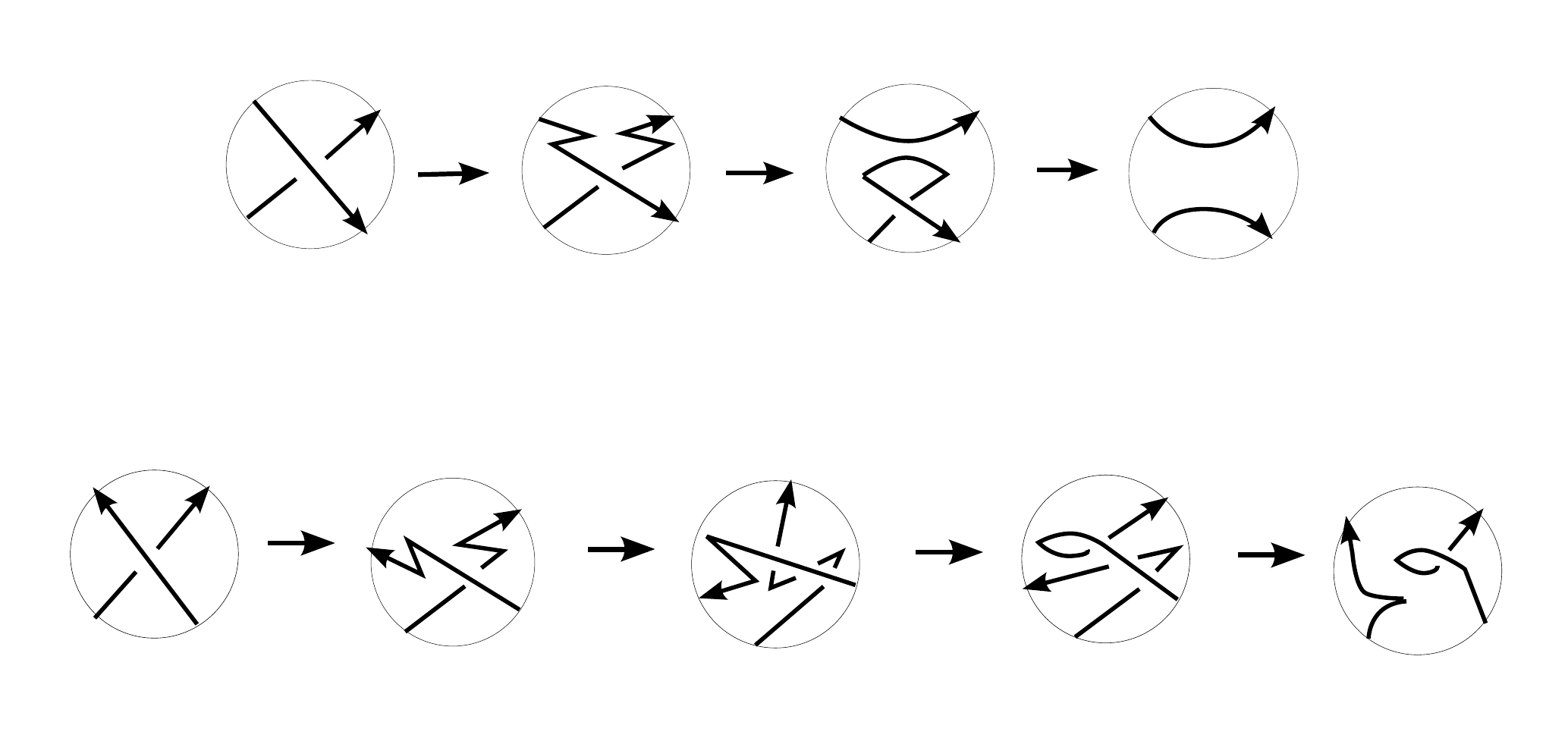
     \caption{The two cases in Step $1$.}
\end{figure}
In the first case one adds two stabilizations, uses the move $H$ and then $R_1$ to resolve the crossing, while in the second case one adds two stabilizations and then performs an $R_2$ move, a handle attachment and another $R_2$ move to simplify the crossing. It is clear that after performing these moves to every crossing one obtains a Legendrian link as the claimed one.
\\
\par
\textit{Step 2. }We show that we can reduce to a standard unknot $K(n,m)$ as in Figure $7$. Here $m$ is number of descending cusps pointing to the right, and $n$ is the number of ascending cusps pointing to the left (for example, the knot in the figure is $K(2,1)$).
\begin{figure}[here]
  \centering
\def\svgwidth{0.8\textwidth}
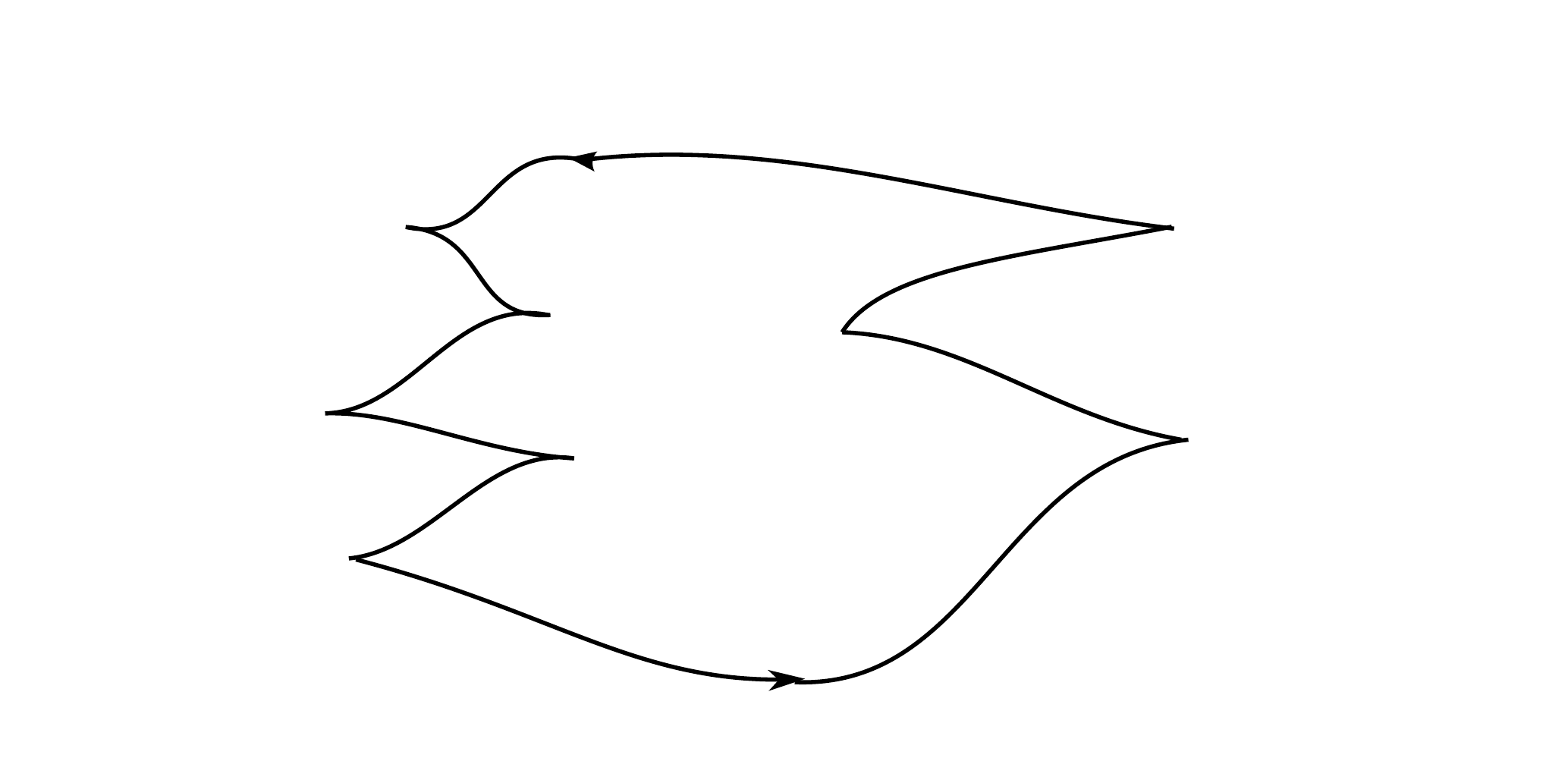
     \caption{A standard form for Legendrian unknots.}
\end{figure}
In fact, it is a classical consequence (see for example \cite{Gei}) of the Thurston-Bennequin inequality and the Eliashberg-Fraser Theorem (see \cite{EF}) that every Legendrian unknot is Legendrian isotopic to some $K(n,m)$ (here in order to construct this isotopy one might need to use some Legendrian $R_3$ move). Now, using the fact that in our front projection each component is cointained in a disk disjoint from the other ones, we can put every component of our link in such a form. Then, using an $H$ move we can connect sum them along the cusps as in Figure $8$ (notice that the orientations match) in order to obtain a single unknot, which can be also isotoped to a standard model.
\begin{figure}[here]
  \centering
\def\svgwidth{0.8\textwidth}
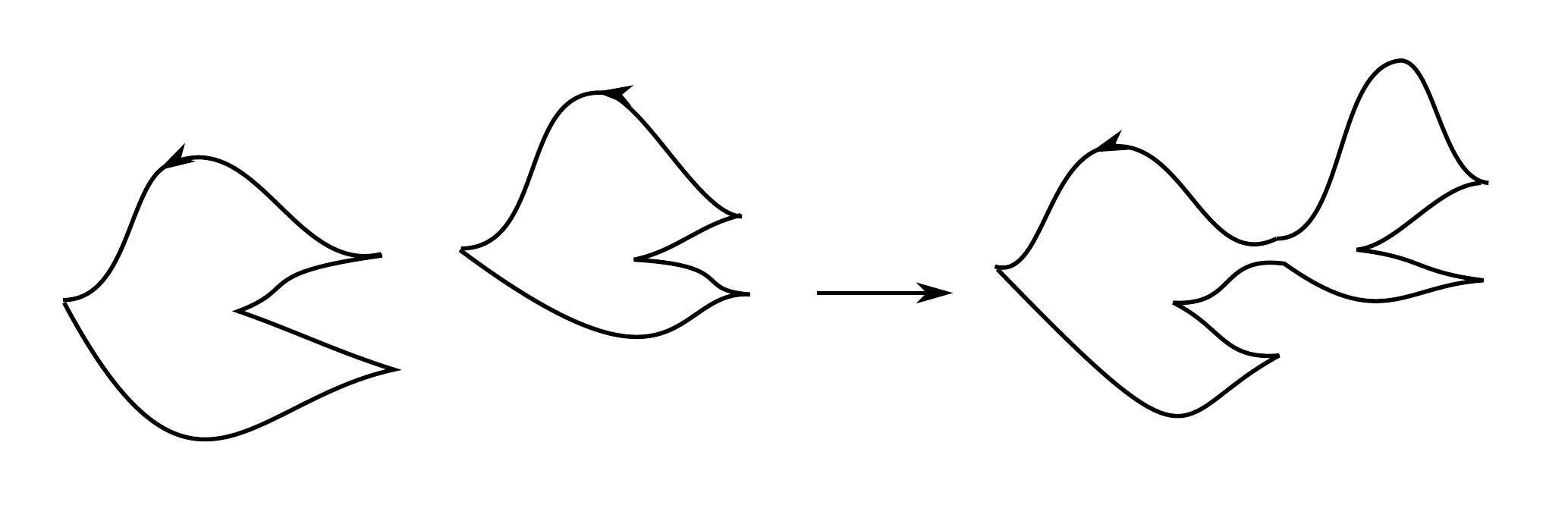
     \caption{Connected sum of $K(m,n)$ knots along cusps.}
\end{figure}
\\
\par
\textit{Step 3. }We construct a cobordism to $U_0$. After adding some stabilizations, we can suppose our knot is a $K(n,n)$ with $n\geq 1$. Our goal is to construct a cobordism to $U_0=K(1,1)$ and the proof will follow by induction if we prove that there is a cobordism from $K(n,n)$ to $K(n-1,n-1)$. In order to do so, we can perform an $H$ move on the top pair of cusps pointing inward, obtaining a $K(0,0)$ and a $K(n-1,n-1)$. Then, we can connect sum these two knots along the cusps, and obtain a $K(n-1,n-1)$.
\end{proof}
\newpage
The same ideas can be exploited to show the following easy lemma.
\begin{lemma}
Given a Legendrian knot $K_0$, there is an exact Lagrangian cobordism from its double stabilization with the same rotation number to itself.
\end{lemma}
\begin{proof}
See Figure $9$. Here the last move is a connected sum along a cusp as is Steps $2$ and $3$ in the proof of the previous proposition.
\end{proof}
\begin{figure}[here]
  \centering
   \def\svgwidth{0.8\textwidth}
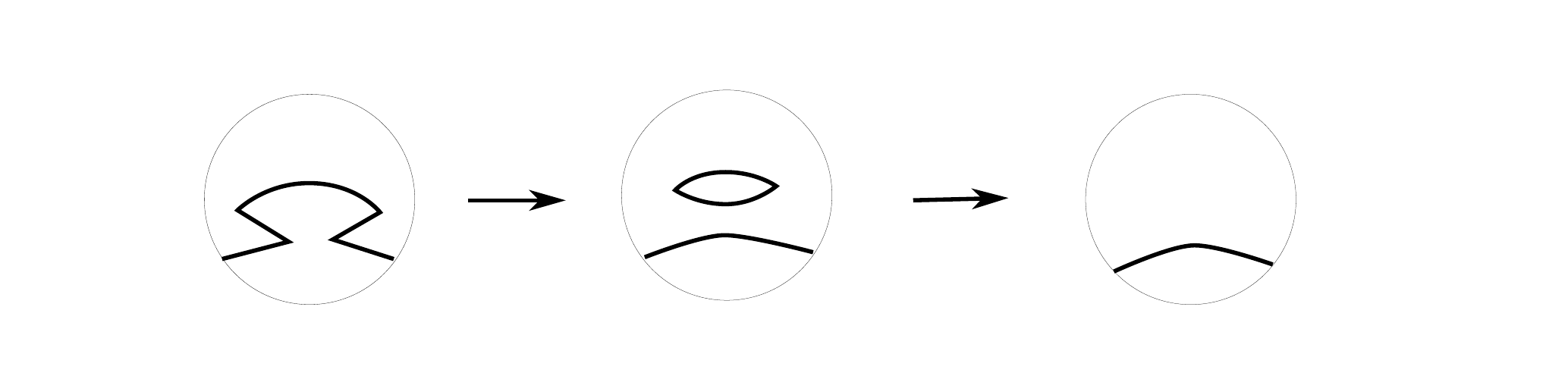
    \caption{A cobordism from the double stabilization.}
\end{figure}

\vspace{1cm}
\section{The combinatorial moves}
As observed in the introduction, in order to construct caps we necessarily have to leave the Legendrian world, and allow more general Lagrangian cobordisms in which the slices are not necessarily Legendrian knots. In this section we show how to construct these, and we first start defining the class of combinatorial objects we will be dealing with.

\begin{defn}
A \textit{Lagrangian diagram} $D$ is a regular projection of a smooth link $K\subset \Rt$ to the $xy$-plane, and the assignment to each domain of the complement of the projection of its area (with respect to the standard area form).
\end{defn}
For example, Figure $2$ (on the right) and Figure $3$ are Lagrangian diagrams of an unknot. It is important to note that the diagram $D$ has to be the projection of some link $K\subset \Rt$.
In order to develop a combinatorial model, we introduce an identification between Lagrangian diagrams as follows.
\begin{defn}
Given a Lagrangian diagram $D$, define $\Delta(D)$ to be the set of positive real numbers arising as the collection consisting of all areas in appearing in the diagram and all absolute values of differences between distinct areas. We say that another Lagrangian diagram $D'$ is \textit{commensurate} to $D$ if the two projections are isotopic as knot diagrams in $\mathbb{R}^2$, and the difference between the two areas assigned to a same domain is strictly smaller than any element in $\Delta(D)\cup\Delta(D')$.
\end{defn}
\begin{notation}
Indeed, we can extend Definition $1$ allowing for a domain to be assigned the number $0$, which means that its area is much smaller than every non-zero area, and every absolute value of the difference of two distinct areas.
\end{notation}
For example, the Lagrangian diagram of Figure $2$ is commensurate to the one on the left in Figure $10$, but not to the one on the right. 
\begin{figure}[here]
  \centering
   \def\svgwidth{0.8\textwidth}
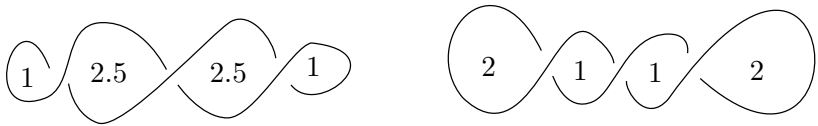
    \caption{Non equivalent Lagrangian diagrams.}
\end{figure}

We now state the main result of the section stating that some combinatorial moves on Lagrangian diagrams correspond to Lagrangian cobordisms between the respective knots.

\begin{prop}Any sequence of local moves as in Figure $11$ between Lagrangian diagrams $D_0$ and $D_1$ with $A,\delta>0$ can be realized by an orientable Lagrangian cobordism in the following sense. Suppose we have a link $K_0$ with Lagrangian diagram $D_0$. Then there exists a link $K_1$ with Lagrangian diagram commensurate to $D_1$ and an orientable Lagrangian $L$ properly embedded in $\Rt\times[0,\varepsilon]\subset(\Rf_+,\omega_{std})$ for some $\varepsilon>0$ such that
\begin{IEEEeqnarray*}{c}
L\cap \{t=0\}=K_0\subset\Rt\\
L\cap \{t=\varepsilon\}=K_1\subset\Rt.
\end{IEEEeqnarray*}
\end{prop}
\begin{figure}[here]
  \centering
   \def\svgwidth{0.8\textwidth}
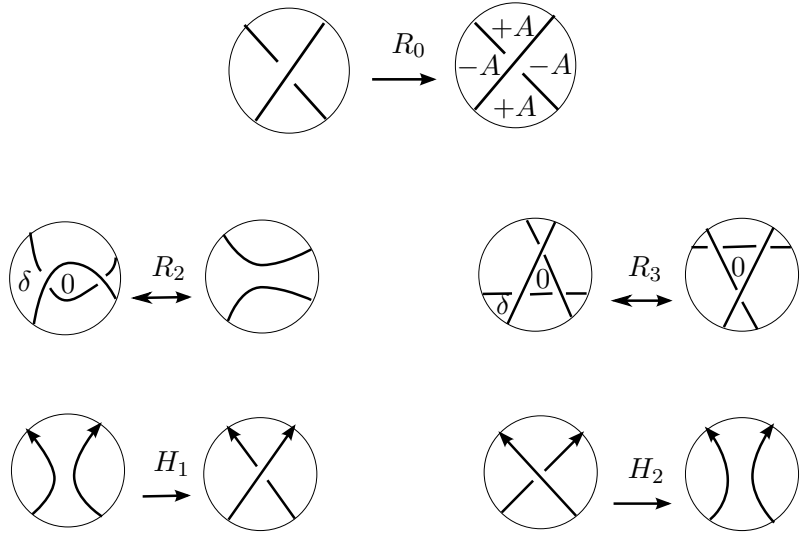
    \caption{Combinatorial moves on Lagrangian diagrams.}
\end{figure}

We spell out the meaning of Figure $11$ in more detail. If a domain is not assigned any number then its area remains unchanged when the move is performed. On the other hand, it will be clear from the construction that the actual areas of the domains in the projection do change, but this change can be made arbitrarily small (hence yielding a knot with a Lagrangian diagram commensurate to the one we want to obtain).
\par
The move $R_0$ (defined for $A>0$) is allowed only if the areas of the domains with label $-A$ are bigger than $A$. The $\delta>0$ in the diagrams for $R_2$ and $R_3$ indicates the necessity of the indicated domain to have area strictly bigger that the one labeled with $0$ (and the move can be performed in both cases also in the case that the symmetric domain is labeled with a $\delta$). Also, the moves $R_2$ and $R_3$ work with any choices of the crossings and are reversible (as the double arrows suggest).
On the other hand, the remaining moves can be performed only with the crossings shown in Figure $11$, and are not reversible. Notice that in the moves $H_1$ and $H_2$ we deal with oriented diagrams. This is done in order to construct orientable cobordisms.
\\
\par
The rest of the section is dedicated to the construction combinatorial moves described in Proposition $2$. The idea is to first construct a Lagrangian cobordism $L_0$ defined for some small $\varepsilon>0$ which is topologically a product
\begin{equation*}
L_0\cong K\times [0,\varepsilon]\subset \Rt\times [0,\varepsilon],
\end{equation*}
such that the Lagrangian diagram of $K\times \{\varepsilon\}$ is commensurate to the one of $K\times\{0\}=K_0$, and then modify such a cobordism in a small neighborhood of the crossing interested by the local move. The whole construction is based on the following key lemma which appears already in \cite{Sau} and describes a way to construct Lagrangian cobordisms starting from a family of immersed curves in the plane.

\begin{lemma}
Suppose we are given a smooth map
\begin{equation*}
\varphi(t,\vartheta)=(x(t,\vartheta), y(t, \vartheta)): [0,T]\times S^1\rightarrow \mathbb{R}^2
\end{equation*}
such that $\varphi(t,-):S^1\rightarrow \mathbb{R}^2$ is an immersion with total signed area $C$ fixed with respect to $t$. Then there exists $z: [0,T]\times S^1\rightarrow \mathbb{R}$ such that the map
\begin{IEEEeqnarray*}{c}
\tilde{\varphi}:[0,T]\times S^1\rightarrow \Rf_+ \\
(t,\vartheta)\mapsto\left(x(t,\vartheta),y(t,\vartheta),t,z(t,\vartheta)\right)
\end{IEEEeqnarray*}
is a Lagrangian immersion. Furthermore, the map $\tilde{\varphi}$ is an embedding if and only if whenever $\varphi(t,\vartheta)=\varphi(t,\vartheta')$, $z(t,\vartheta)\neq z(t,\vartheta')$.
\end{lemma}
\begin{proof}
By imposing $\mathrm{Im}\tilde{\varphi}$ to be isotropic one obtains
\begin{equation*}
0=\tilde{\varphi}^*\omega(\partial /\partial t,\partial/\partial\vartheta)=\frac{\partial x}{\partial{t}}\frac{\partial y}{\partial \vartheta}-\frac{\partial x}{\partial{\vartheta}}\frac{\partial y}{\partial t}+\frac{\partial z}{\partial \vartheta}.
\end{equation*}
Setting
\begin{equation*}
J(t,\vartheta)=-\frac{\partial x}{\partial{t}}\frac{\partial y}{\partial \vartheta}+\frac{\partial x}{\partial{\vartheta}}\frac{\partial y}{\partial t},
\end{equation*}
 if for every $t$ one has
\begin{equation*}
\int_{S^1}J(s,\vartheta)d\vartheta=0
\end{equation*}
then one can define a map $\tilde{\varphi}$ with isotropic image by setting
\begin{equation*}
z(t,\vartheta)=z(t,\vartheta_0)+\int_{\vartheta_0}^{\vartheta}J(t,\vartheta)d\vartheta
\end{equation*}
where $\vartheta_0\in S^1$ is a basepoint and $z(-,\vartheta_0)$ is any fixed function. On the other hand
\begin{multline*}
0=\frac{d}{dt}C=\frac{d}{dt}\int_{S^1}\varphi_t^*(xdy)=\int_{S^1}\mathcal{L}_{\partial/\partial t}\varphi^*(xdy)=\\ =\int_{S^1}d(\varphi^*i_{\partial/\partial t} xdy)+ \int_{S^1}i_{\partial/\partial t} dx\wedge dy= \int_{S^1}\frac{\partial x}{\partial{t}}\frac{\partial y}{\partial \vartheta}-\frac{\partial x}{\partial{\vartheta}}\frac{\partial y}{\partial t}d\vartheta.
\end{multline*}
From the shape of $d\tilde{\varphi}$ one sees that if each curve $\varphi_t$ is immersed then $\tilde{\varphi}$ will be an immersion. Finally, the embeddedness criterion is straightforward.
\end{proof}

\begin{proof}[Proof of Proposition 2] As previously discussed, we first construct a Lagrangian cobordism $L_0$ between the given knot $K_0$ and a new knot with Lagrangian diagram commensurate to $D_0$. This cobordism will be defined for some small $\varepsilon>0$ and will be topologically a product. Indeed, the (parametrized) knot $K_0$ comes with the function
\begin{equation*}
J_0(\vartheta)=dz/d\vartheta.
\end{equation*}
One can find a family of immersions $\varphi: [0,1]\times S^1\rightarrow \mathbb{R}^2$ with constant area such that
\begin{itemize}
\item the restriction $\varphi(0,-)$ is the parametrization of $K_0$;
\item for the function $J$ defined in Lemma $2$, $J(0,-)=J_0$.
\end{itemize}
Notice that the function $J$ can be interpreted as the scalar product between $\partial\varphi/\partial\vartheta$ and the $\pi/2$ clockwise rotation of $\partial\varphi/\partial s$ so determines only the component of $\partial \varphi/\partial t$ normal to the curve, hence one can actually arrange the total change of area to be zero.
\par
After restricting to a smaller time interval $[0,\varepsilon]$, such a map induces via the construction of Lemma $2$ an embedded Lagrangian cobordism with the property that all the slices with $s$ constant have Lagrangian diagram commensurate to $D_0$. In order to prove Proposition $2$, we modify this family of immersions $\varphi$ in a neighborhood of the crossing where each local move is performed as follows.
\\
\par
\textit{Move $R_0$. }This move can be achieved using a family of immersions $\psi$ as in Figure $12$. 
\begin{figure}[here]
  \centering
   \def\svgwidth{0.8\textwidth}
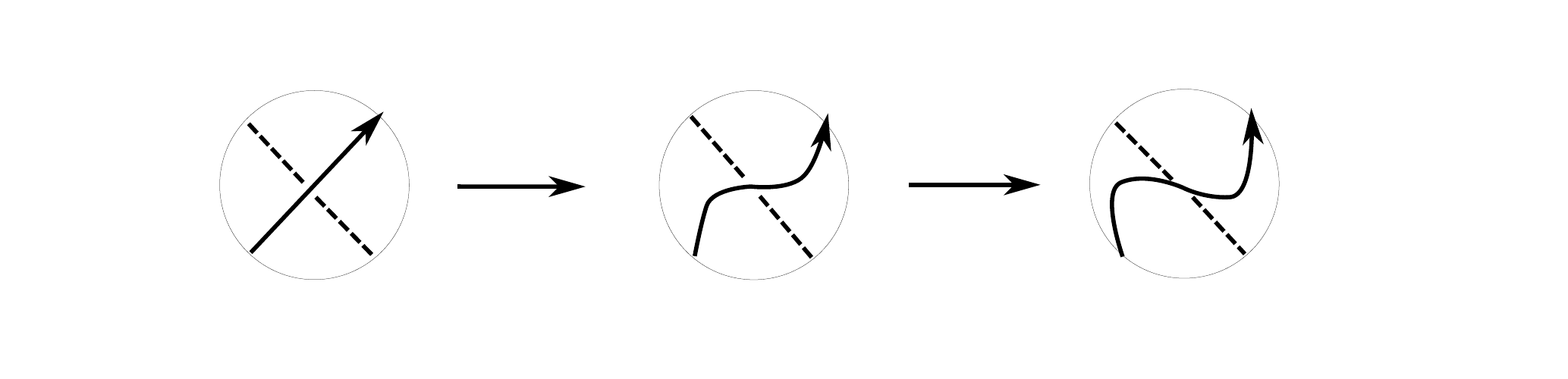
     \caption{Constructing the move $R_0$.}
\end{figure}

To spell out the meaning of the picture, it refers to a neighborhood of the crossing where the areas of the domains on which the move decreases the area is bigger that $A$. In agreement with our convention, we schematically represent the map $\varphi$ we start with as the constant diagram on the left of the picture. We then modify such a map to a new family $\psi=\{\psi_t\}$ shown in the picture. Such a family agrees with $\varphi_t$ near the boundary of the neighborhood for $t\in[0,\varepsilon]$ and for $t\geq0$ small in the whole neighborhood. It is clear that such a family $\psi_t$ can be chosen so that the total area remains constant. Furthermore, we can arrange that the family $\psi$ still defines an embedded Lagrangian cobordism. This is because our family $\psi$ can be chosen so that it respects the previous requirements and is such that the integral of its $J$ function along the solid arc is bigger for every $t$ than the corresponding integral for $\varphi$ before the crossing. This implies that at the crossing the $z$ coordinate of the new induced knot will be bigger that the starting one, hence in particular that the cobordism is embedded. Finally, it is clear that the family can also be chosen so that the area of the domains changes by $A$.
Notice that the specular construction of the similar move with $A<0$ will in general give rise to a Lagrangian with a self-intersection.
\\
\par
\textit{Move $R_2$. }The construction is achieved by the family of immersions depicted in Figure $13$ (here the orientations are introduced only to describe the family). In particular, the part of the of the arcs before the first crossing is pulled to the left, while a neighborhood of the arcs connecting the two crossing points are pulled down (for the lower arc) and up (for the upper arc). Such a family is chosen as before so that the total area remains constant  and coincides with $\varphi_t$ near the boundary of the neighborhood for $t\in[0,\varepsilon]$ and for $t\geq0$ small in the whole neighborhood. We are assured that such a family exists because the area enclosed by the crossing is much smaller that the area of the domain towards which we are pulling the arcs. The embeddedness of the induced cobordism can be arranged exactly as in the construction of move the $R_0$. The inverse move is constructed using an analogous family.
\begin{figure}[here]
  \centering
   \def\svgwidth{0.8\textwidth}
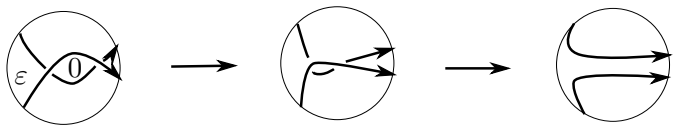
     \caption{Constructing the move $R_2$.}
\end{figure}
\\
\par
\textit{Move $R_3$. }The construction of this move is totally analogous to the previous ones.
\\
\par
\textit{Move $H1$. }The handle attachment moves cannot be achieved by the same construction, as the topology of the cobordism changes. In order to construct them, we can rely on the local model shown in Figure $14$ provided by a neighborhood of the Lagrangian submanifold
\begin{equation*}
\{(x,y,t,z)\mid t=z^2-y^2, x=2yz\}\subset \Rf
\end{equation*}
around the origin. Alternatively, one can also rely on the handle attaching moves constructed in Section $1$ using front projections and jet spaces.
\begin{figure}[here]
  \centering
 \def\svgwidth{0.8\textwidth}
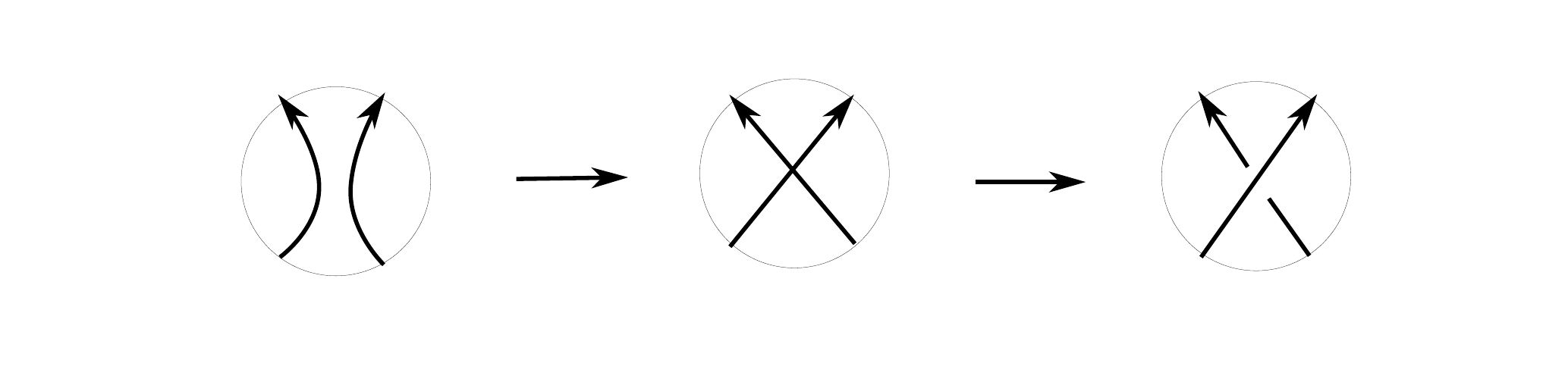
     \caption{The local model for a handle attachment.}
\end{figure}
As before, we describe how the local move is performed in a small neighborhood of the crossing affected by the local move. Suppose we start with a family of immersions $\{\varphi_t\}$ defined for a time $\varepsilon>0$. After restricting to smaller arcs, we can obtain a local model for the handle attachment $\tilde{L}$ that can be realized in a time $\varepsilon/2$ (indeed, we just need to restrict  the Lagrangian we start with to a smaller time interval). Pick a neighborhood $\mathcal{U}$ of the crossing interested by the local move. Then, in time $\varepsilon/2$ we can construct an isotopy of the arcs inside this neighborhood so that the induced arcs in $\Rt$ agree with the lower boundary of our local attaching model $\tilde{L}$ in a smaller neighborhood $\mathcal{U}'$. Clearly such a family can be chosen so that it preserves the area and coincides with $\{\varphi_t\}$ near the boundary of $\mathcal{U}$. We just then attach $\tilde{L}$ inside $\mathcal{U}'$ and modify the family of immersions for $t\in[\varepsilon/2,\varepsilon]$ in $\mathcal{U}\setminus \mathcal{U}'$ so that together they define a smooth embedded Lagrangian for $t\in[0,\varepsilon]$.
\\
\par
\textit{Move $H_2$. }The move $H_2$ can be obtained as a combination of the previous moves as in Figure $15$. In particular, one performs a $H_1$ move in a small neighborhood above the crossing such that the resulting pair of crossing encloses a very small area and then an $R_2$ move.
\begin{figure}[here]
  \centering
 \def\svgwidth{0.8\textwidth}
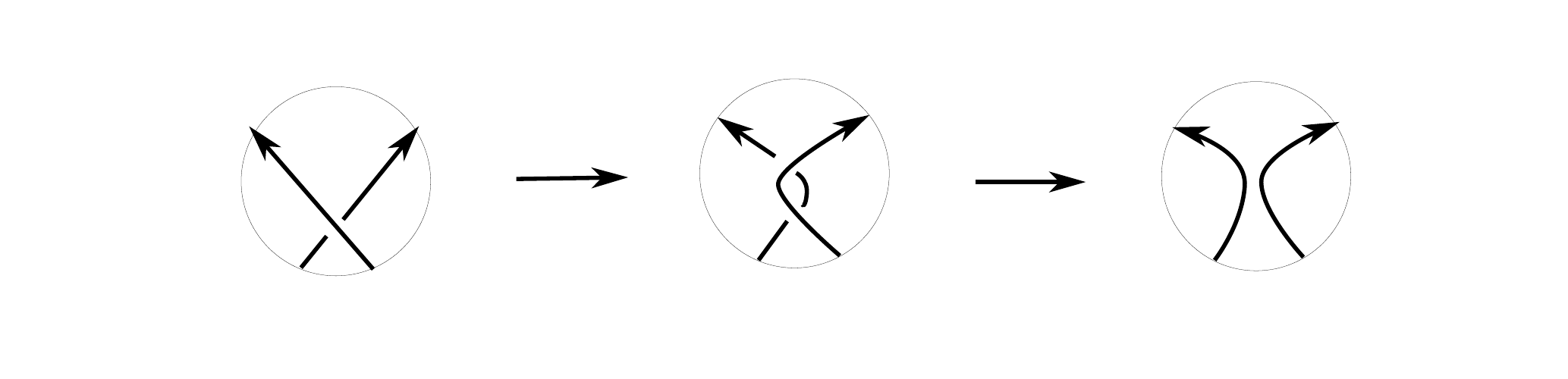
     \caption{Constructing the move $H_2$.}
\end{figure}

Finally, it is clear that given a sequence of moves each Lagrangian cobordism can be constructed so that the total result is smooth. This concludes the proof.
\end{proof}

We conclude this section by constructing caps for the knots with Lagrangian diagram $U_1$ and fillings for the knots with the mirror Lagrangian diagram (i.e. the  Lagrangian diagram with the opposite crossing, which is also the Lagrangian projection of the standard Legendrian unknot).
\begin{lemma}
Any knot with Lagrangian diagram commensurate to $U_1$ admits a Lagrangian cap which is topologically a disk. Similarly, any knot with Lagrangian diagram commensurate to the mirror of $U_1$ admits a Lagrangian filling which is topologically a disk.
\end{lemma}
\begin{proof}
In order to construct a cap we simply rely on a suitable translation of the Lagrangian open disk parametrized in polar coordinates by
\begin{IEEEeqnarray*}{c}
\psi:\mathbb{R}^2\rightarrow \Rf \\
(r,\vartheta)\mapsto \left(r\cos\vartheta, r\sin2\vartheta, -2r, r(\sin\vartheta-\frac{1}{3}\sin^3\vartheta)\right).
\end{IEEEeqnarray*}
Indeed, using the same techniques as in the proof of Proposition $2$, we can construct a Lagrangian cobordism (which is topologically a cylinder) from the given knot to (a suitable translation) of one of the curves $\psi(r_0,-)$, and then cap the latter using the image via $\psi$ of the disk of radius $r_0$. The construction of the filling is analogous, as one can use the map obtained from $\psi$ by changing sign to the last two coordinates.
\end{proof}

\vspace{1cm}
\section{An exactness criterion}
So far, we have constructed embedded Lagrangians which are not necessarily exact, and it is well known that the last condition is a very restrictive one. For instance it is easy to construct non-exact closed embedded Lagrangians in $(\Rf,\omega_{std})$, while a celebrated result by Gromov (\cite{Gro}) tells us that there cannot be any exact closed embedded Lagrangian. Indeed, to deal with exactness there are some some subtle global aspects to be considered. First, we introduce a definition.
\begin{defn}
Given a Lagrangian cobordism $L\subset\Rf_+$, we define its \textit{negative boundary} to be $K_0=L\cap\{ t=0\}$. If the negative boundary $K_0$ of $L$ is Legendrian we say that $L$ is \textit{exact relative to its negative boundary} if $\lambda_{std}\lvert_L$ has a primitive that vanishes on $K_0$.
\end{defn}
Clearly, $L$ is exact relative to its negative boundary if and only if the integral of $\lambda$ is zero on every closed curve and on every embedded arc with both endpoints on the negative boundary $K_0$. Furthermore, in the case $L$ is a cap of $K_0$ then $L$ is exact in the usual sense. The primitive
\begin{equation*}
\lambda_{std}=xdy+(t+1)dz
\end{equation*}
is particularly nice because its integral along a curve with constant $t$ coordinate is exactly the area enclosed by its projection to the $xy$-plane. For example, the sequence of moves in Figure $16$ corresponds to a Lagrangian torus which is manifestly non-exact, and is indeed one of the standard Clifford tori. Here, the dashed arrow on the left indicates the standard filling constructed in Lemma $3$, and is followed by an $H_2$ move, an $H_1$ move and finally one uses Lemma $3$ again to construct the cap.
\begin{figure}[here]
  \centering
 \def\svgwidth{0.8\textwidth}
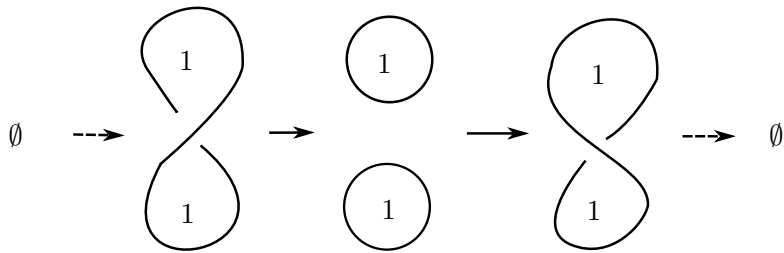
     \caption{A non-exact torus in $\mathbb{R}^4$.}
\end{figure}

Before stating our exactness criterion, we first  introduce a simple definition.
\begin{defn}
Given a projection of a link $K$, we say that two components of the link $K_1,K_2$ are \textit{vertically split} if the image of the projection can be covered with two disjoint disks $\Delta_1,\Delta_2$ with $K_i\subset\Delta_i$ for $i=1,2$.
\end{defn}

\begin{prop}
Suppose we are given a Legendrian link $K_0$ with Lagrangian diagram $D_0$ and a sequence of Lagrangian diagrams related by combinatorial moves starting at $D_0$ such that
\begin{itemize}
\item at each stage every component of the link has total area zero;
\item if a handle attachment move which merges two components is performed, then these components are vertically split for the projection to the $xy$-plane.
\end{itemize}
Then the sequence of combinatorial moves can be realized by a Lagrangian cobordism which is exact relative to its negative boundary $K_0$.
\end{prop}

The conditions on the sequence of diagrams imposed by this proposition are rather strict, but they are general enough to be applied to our case. On the other hand, there are many subtleties that come in play when one tries to generalize this result. For example, the sequence of moves in Figure $17$ corresponds to a non-exact Lagrangian torus in $\mathbb{R}^4$ such that all the slices encloses area zero. More in detail, after the usual filling provided by Lemma $3$ one performs an $R_2$ move, an $R_0$ move, an $H_1$ move, an $R_2$ move, an $H_2$ move, an $R_0$ move, an $R_2$ move, an $R_0$ move, an $R_2$ move and finally one uses the cap from Lemma $3$. It is easy to check that all these moves are compatible with the orientation of the starting unknot.
\begin{figure}[here]
  \centering
 \def\svgwidth{0.8\textwidth}
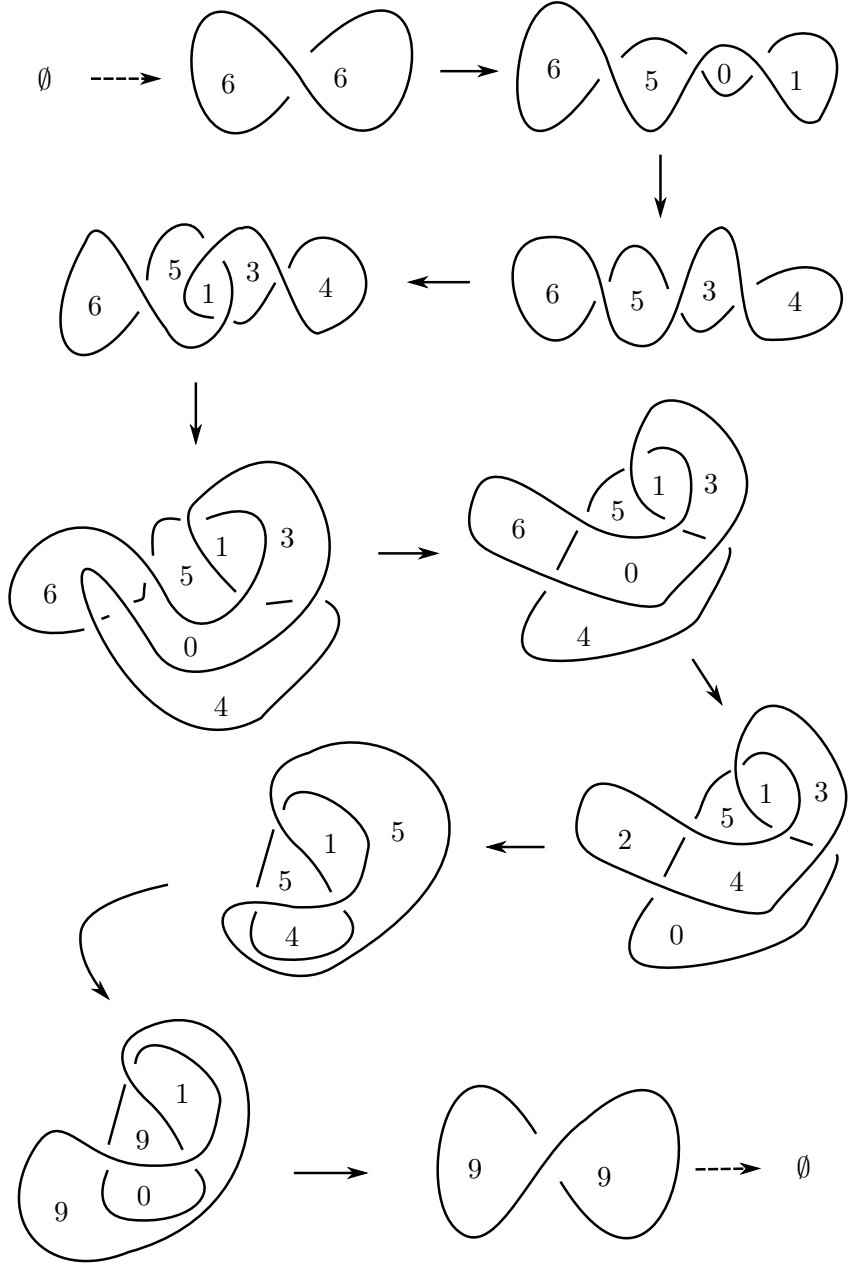
     \caption{A non-exact torus with all enclosed areas zero.}
\end{figure}
\\
\par
Indeed, we will prove that a sequence of Lagrangian diagrams related by a sequence of combinatorial moves satisfying the conditions of Proposition $3$ can be realized by a Lagrangian cobordism such that the integrals of the primitive $\lambda$ on some fixed curves can be made arbitrary. In particular, for different choices of these integrals we obtain non-Hamiltonian isotopic Lagrangians.

\begin{proof}[Proof of Proposition 3]
The exactness of a Lagrangian cobordism $L$ relative to its negative boundary can be showed just by checking that the integral of $\lambda$ along the well-chosen collection of closed curves and arcs $\{\alpha_i\},\{\beta_i\}$ and $\{\gamma_i\}$ displayed in Figure $18$. In particular the collection consists of the following:
\begin{itemize}
\item for every handle attachment move that creates a new component of the link, a closed curve $\alpha_i$  for which the $t$ coordinate is constant;
\item for every handle attachment move that merges two components of the link, a closed curve $\beta_i$;
\item for every handle attachment move that creates a new component of the link, an embedded arc  $\gamma_i$ with boundary on $K_0$.\end{itemize}

\begin{figure}[here]
  \centering
\def\svgwidth{0.8\textwidth}
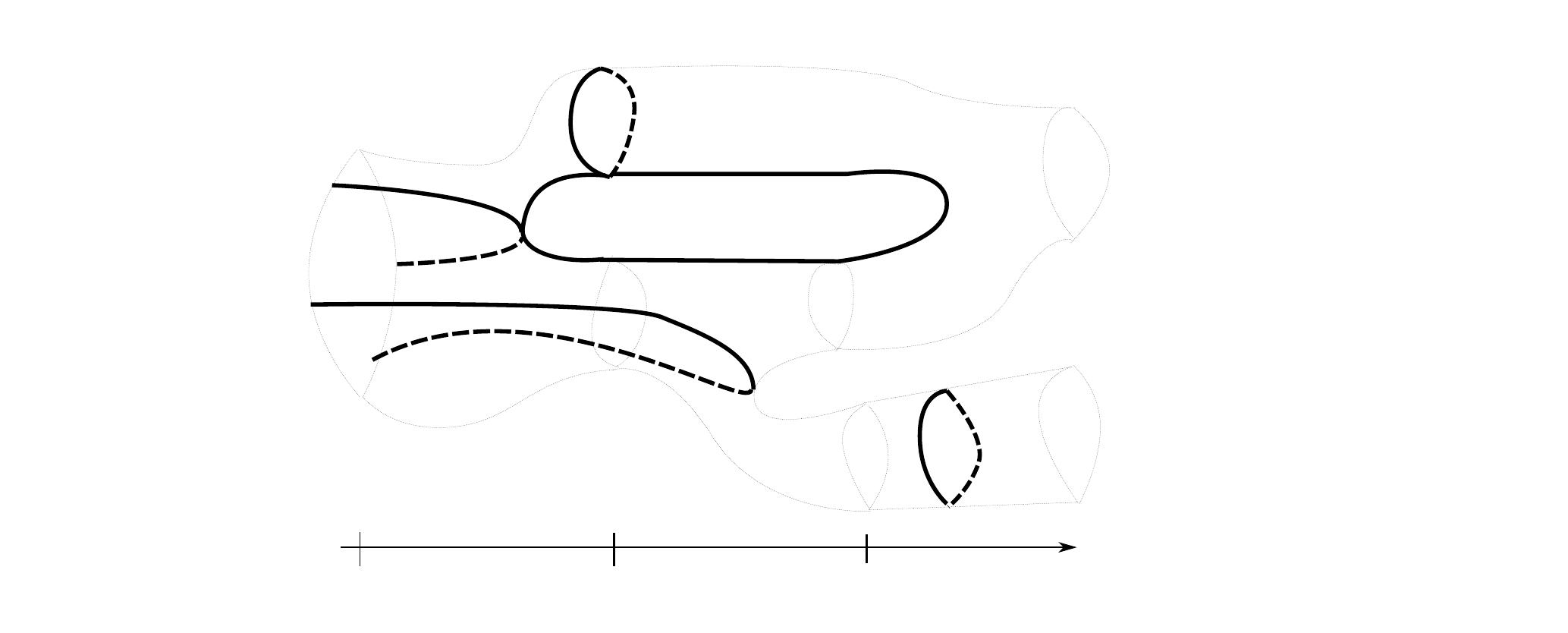
     \caption{The set of curves and arcs that we need to .}
\end{figure}
The integral condition on the slices with $t$ constant is exactly the fact that at each stage every component of the link has total area zero in the hypothesis of the proposition. Indeed, it is easy to see that our Lagrangian can be constructed so that the integral along each these curves is exactly zero. Hence we just have to show that it is possible to construct a Lagrangian such that the integral along each of the curve $\beta_i$ and $\gamma_i$ is zero. Suppose then that we are given a handle attachment cobordism which merges two components of the link, and suppose $\beta_i$ is the corresponding curve. We show how to modify the construction of such a cobordism exploiting the freedom of choice of the function $z(s,\vartheta_0)$ in Lemma $2$.  This will change the handle attachment, but will preserve the Lagrangian projections. Looking again at the construction of the move $H_1$ in the proof of Proposition $2$, before the local model for the handle attachment is glued we have a Lagrangian cobordism (the tube on the left in Figure $19$) induced by a family of immersions
\begin{equation*}
\varphi:[0,\varepsilon/2]\times S^1\rightarrow \mathbb{R}^2.
\end{equation*}
We can suppose without loss of generality that $\tilde{\varphi}(t,\vartheta_0)$ gives a parametrization of the solid subarc $\gamma\subset \beta_i$ shown in Figure $19$ for $t\in[0,\varepsilon/2]$. Now any choice of the height function $z(t,\vartheta_0)$ such that its value at the agrees with the one we started with near the endpoints will give rise to a cobordism that agrees with the original one outside this small tube. Because of the vertical split hypothesis in the statement of the proposition, by modifying accordingly also the $z$ functions of the other components of the link which project to the same disk in the $xy$-plane we can obtain an embedded cobordism. Furthermore, one has
\begin{equation*}
\int_{\gamma} \lambda=\int_0^{\varepsilon/2}\left[x(t)y'(t)dt+(t+1)z'(t,\vartheta_0)\right]dt
\end{equation*}
which can assume any desired value for an appropriate choice of the function $z(t,\vartheta_0)$. Hence we can arrange that the a new Lagrangian cobordism is such that the total integral along $\beta_i$ is any value we want, and in particular it can made exactly zero.
\begin{figure}[here]
  \centering
\def\svgwidth{0.8\textwidth}
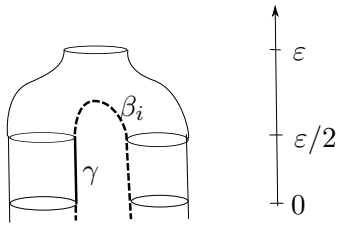
     \caption{Perturbing the family of immersions for $s\in[0,\varepsilon/2]$.}
\end{figure}

Finally, the same argument can be used in order to deal with the integral along the arcs $\gamma_i$, which is indeed simple as one can always achieve embeddedness.
\end{proof}

\vspace{1cm}
\section{An exact Lagrangian cap for $U_0$}

In this section we complete the proof of Theorem $1$ by constructing an exact Lagrangian cap for the knot $U_0$ using the set of combinatorial moves we have previously developed. Before describing the construction, we introduce some \textit{curl moves} which will be very useful in the construction. This is because there is not an analogue of the $R_1$ move for Lagrangian diagrams, so we cannot get easily rid of the curls (in fact, it is generally impossible). On the other hand, the curl moves allow us to move these curls around the diagram. 

\begin{lemma} Each of the moves in Figure $20$ can be realized by a Lagrangian cobordism. For $C_1$, the two domains on the right need to have area bigger than $A$, while for $C_2$ the two upper domains need to have area bigger than $A$.
\begin{figure}[here]
  \centering
\def\svgwidth{0.8\textwidth}
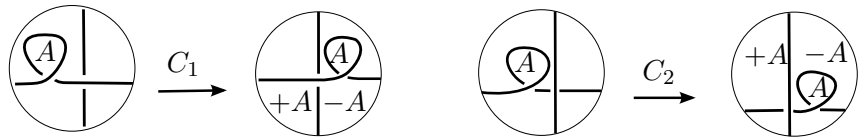
     \caption{The curl moves.}
\end{figure}
\end{lemma}

Clearly there are analogous inverse moves, and there are also versions in which the curl has the opposite crossing.

\begin{proof}We just show how to construct the move $C_1$, as the other is totally analogous. This is explained in Figure $21$.
\begin{figure}[here]
  \centering
\def\svgwidth{0.8\textwidth}
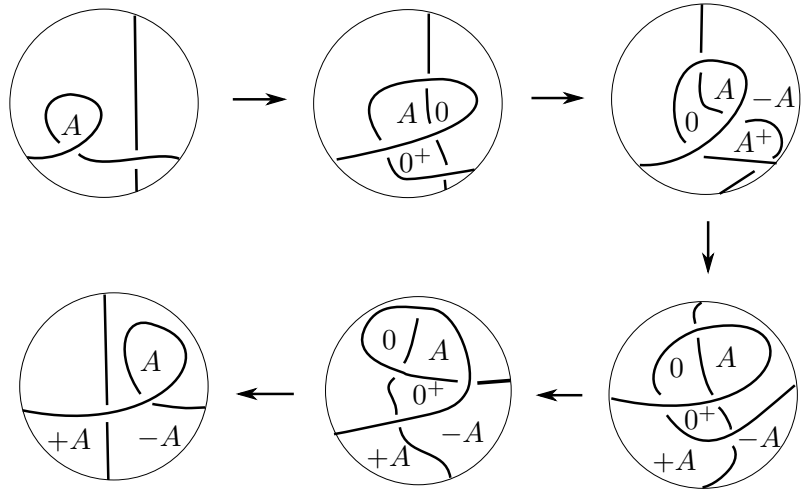
     \caption{Construction of the move $C_1$.}
\end{figure}
\newpage
In particular, the moves are the following:
\begin{itemize}
\item A move $R_2$ such that the new created domains have small area (here $0^+$ indicates a small number bigger that $0$);
\item An $R_0$ move (here we need the upper right domain to have area bigger than $A$, and $A^+$ indicates number slightly bigger that $A$);
\item An $R_0$ move crossing (here we need the lower right domain to have area bigger than $A$);
\item An $R_3$ move;
\item An $R_2$ move.
\end{itemize}
\end{proof}

Our strategy is to construct an exact Lagrangian cobordism to an union of a component $U_1$ as in Figure $3$ together with two standard Legendrian unknots. The first can be easily capped using Lemma $3$. To get rid of the other components, we will use the following \textit{unknot trick} before constructing the final cap.
\begin{lemma}Provided $A>B$, the move of Figure $22$ can be realized by a Lagrangian cobordism.
\end{lemma}
\begin{figure}[here]
  \centering
  \def\svgwidth{0.8\textwidth}
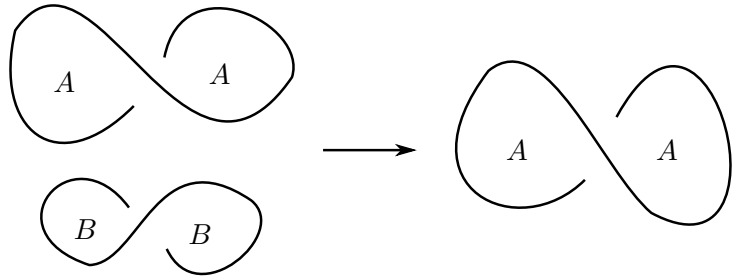
     \caption{The unknot trick.}
\end{figure}
\begin{proof}
The proof is contained in Figure $23$. The first move is an $H_1$ move (it is straightforward to see that this can be done for any choice of the orientations). Then one simply performs an $R_0$ and an $R_2$  move as suggested in the figure.
\begin{figure}[here]
  \centering
\def\svgwidth{0.8\textwidth}
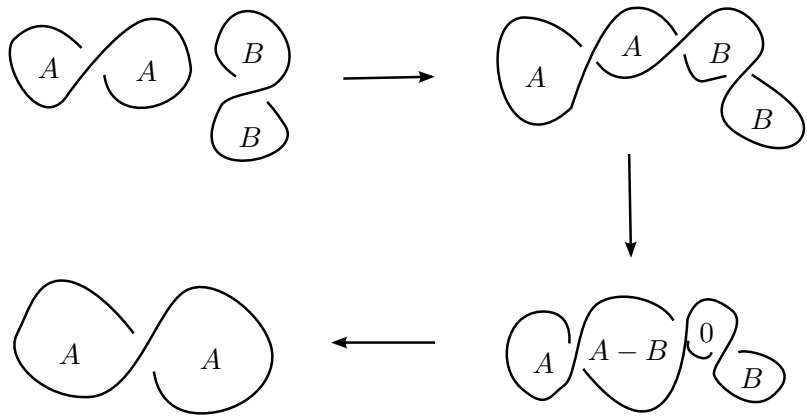
    \caption{Proof of the unknot trick.}
    \end{figure}
\end{proof}

Finally, we are ready to give a proof of Theorem $1$.

\begin{proof}[Proof of Theorem 1]We show how to construct an exact cap for the knot $U_0$. First, one performs the moves displayed in Figure $24$, which are in order an $R_0$, an $R_2$, an $R_0$, a $C_1$, a $C_2$ and an $H_2$. One then focuses on the right component of the result, and use the moves indicated in Figure $25$, which are two $R_0$ moves, two $C_2$ moves and an $H_2$ move. At this point, one can use the unknot trick to get rid of all the Legendrian unknot components, and then cap the remaining one using an $R_0$ move. We hence obtain a Lagrangian cap of the knot $U_0$, and we can perform the construction so that the result is exact by the criterion provided by Proposition $3$. Indeed, no handle attachment moves are used in the construction of the curl moves, and all the handle attachments (including those used for the unknot trick) are performed on vertically split components. Finally, the theorem follows from the existence of this cap for $U_0$ and Proposition $1$ and Lemma $1$ in Section $1$.
\end{proof}

\vspace{1cm}
\section{Further problems}

The techniques used in order to prove Theorem $1$ via the combinatorial moves might be adapted to other problems related to Legendrian knots. On the other hand, even if it is easy to construct a Lagrangian projection of a given Legendrian knot (see for example the algorithm in \cite{Etn}), we are not aware of any easy way to determine a Lagrangian diagram, i.e. a way to assign areas to the domains. This is a big limitation for the techniques developed in the present paper, especially when one tries to work out some practical examples. Indeed, it is already non trivial to understand if a given Lagrangian diagram can be the Lagrangian projection of a Legendrian knot.
\par
Another interesting task is to find a good generalization of the exactness criterion provided by Proposition $3$, and in particular one that fits in our combinatorial model. This might be related to very subtle issues in contact and symplectic geometry, and indeed it is not improbable that one can find a description of a non-exact torus such that all the areas enclosed are zero and the handle attachment that merges two components is performed on two unlinked (but not vertically split) components (making the strict requirements in Proposition $3$ look not that excessive).
\par
Finally, there are two more questions that naturally arises from our treatment.
\begin{question}Is there any Legendrian knot with an exact Lagrangian cap of genus $1$?
\end{question}
While every cap we have constructed in the paper has genus at least $2$, this case is not ruled out by the slice Thurston-Bennequin inequality (while the genus $0$ case is).

\begin{question}Is there any not destabilizable Legendrian knot which admits an exact Lagrangian cap?
\end{question}
This is a very interesting question, as it is a fundamental assumption for Theorem $1$ that we can add as many stabilizations as desired.
\begin{figure}
  \centering
\def\svgwidth{0.9\textwidth}
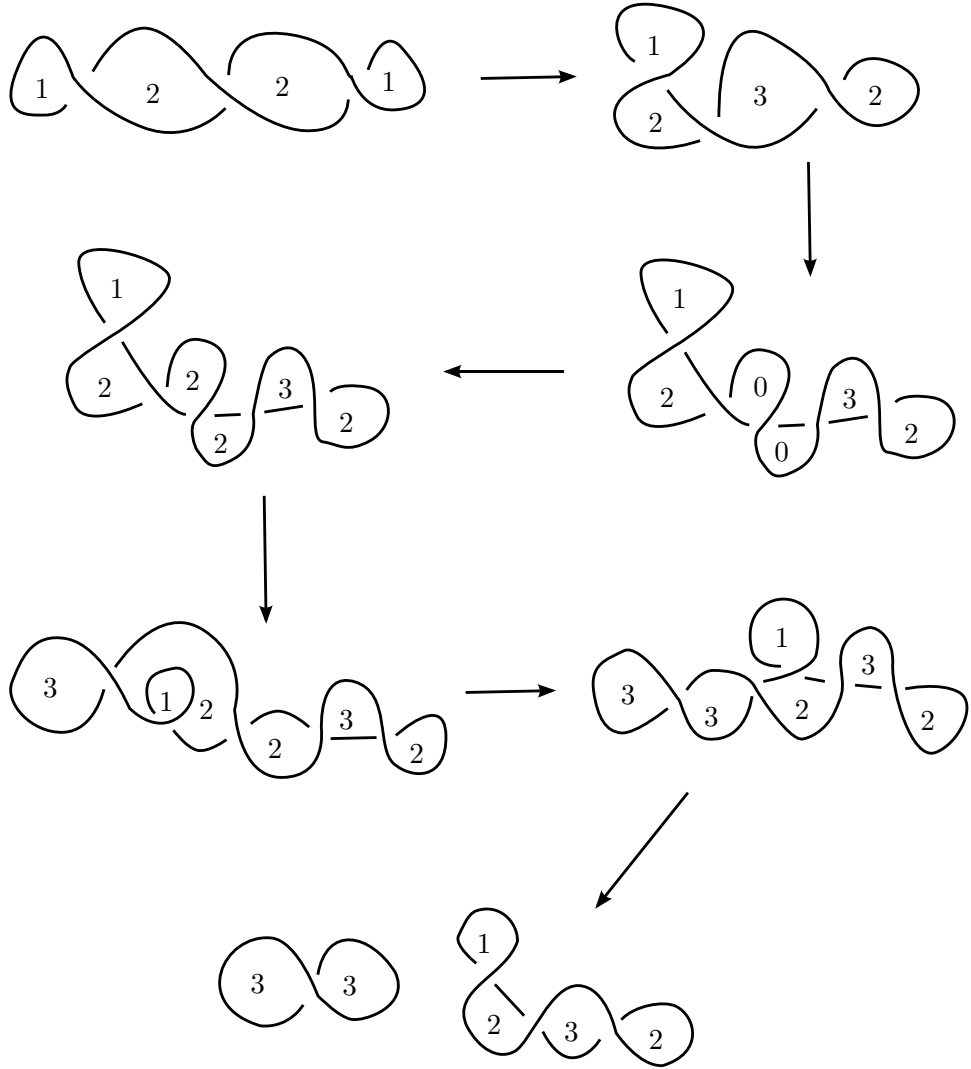
     \caption{Constructing a cap for $U_0$ - part $1$.}
\end{figure}

\begin{figure}
  \centering
\def\svgwidth{0.9\textwidth}
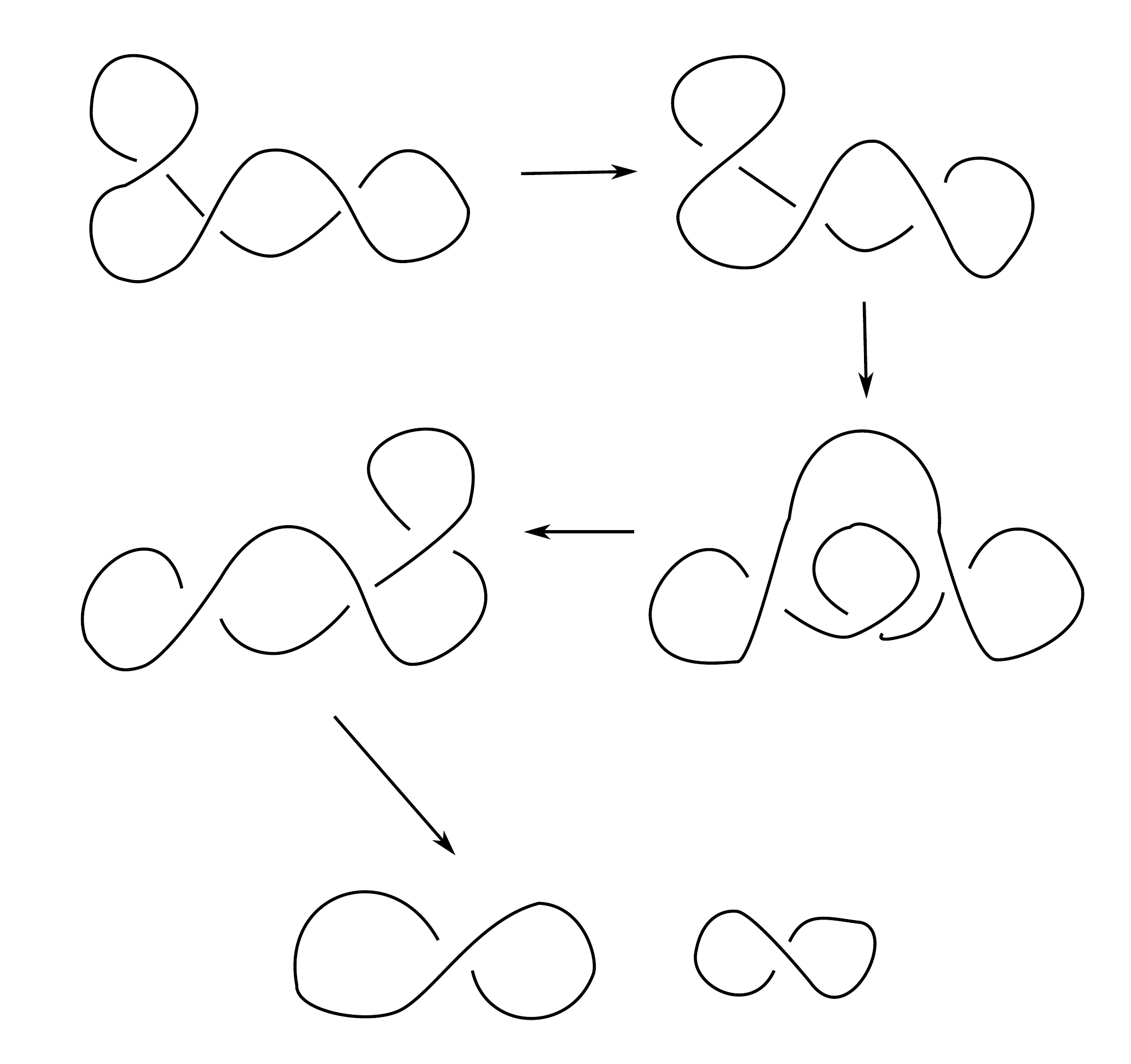
     \caption{Constructing a cap for $U_0$ - part $2$.}
\end{figure}
\clearpage
\bibliographystyle{alpha}
\bibliography{biblio}

\end{document}